\documentclass[11pt]{article}

\usepackage[a4paper]{geometry}
\usepackage
[
colorlinks=true,
linkcolor=blue,
anchorcolor=blue,
citecolor=blue,
urlcolor=blue,
plainpages=false,
pdfpagelabels,
breaklinks=true
]{hyperref}
\usepackage[english]{babel}
\usepackage[utf8]{inputenc}
\usepackage{amsmath,amssymb,amsthm,thmtools}
\usepackage{newtxtext,newtxmath}
\usepackage{graphicx}
\usepackage[colorinlistoftodos,textsize=scriptsize]{todonotes}
\usepackage{fullpage}
\usepackage[margin=0.5in]{caption}
\usepackage{url, inconsolata}
\usepackage{tikz}
\usetikzlibrary{calc}
\usetikzlibrary{decorations.pathreplacing}
\usepackage[capitalize]{cleveref}
\usepackage{xspace}
\usepackage{enumerate}
\usepackage[numbers,sort,compress]{natbib}
\usepackage{mathtools}
\usepackage{algorithm,algpseudocode}
\usepackage{tabularx}
\usepackage{listings}
\usepackage{nicematrix}
\NiceMatrixOptions{
code-for-first-row = \scriptstyle,
code-for-last-row = \scriptstyle,
code-for-first-col = \scriptstyle,
code-for-last-col = \scriptstyle 
}
\DeclareMathOperator{\pivots}{Pivots}
\DeclareMathOperator{\Rips}{Rips}
\DeclareMathOperator{\diam}{diam}

\newcommand\op[1]{\mathop{\operatorname{#1}}\nolimits}
\newcommand\cat[1]{\ensuremath{\mathbf{#1}}}

\newcommand\colpivot[2]{\op{Pivot}{#1_#2}}

\newcommand\colpivotentry[2]{\op{PivotEntry}{#1_#2}}

\newtheorem{theorem}{Theorem}[section]
\newtheorem*{theorem*}{Theorem}
\newtheorem{proposition}[theorem]{Proposition}
\newtheorem{lemma}[theorem]{Lemma}

\theoremstyle{definition}
\newtheorem{definition}[theorem]{Definition}
\newtheorem{remark}[theorem]{Remark}

\title{Ripser: efficient computation of Vietoris–Rips~persistence~barcodes}
\author{Ulrich Bauer\thanks{Technical University of Munich (TUM), Germany.%
}}
\begin{document}
\maketitle

\begin{abstract}
We present an algorithm for the computation of Vietoris–Rips persistence barcodes and describe its implementation in the software Ripser.
The method relies on implicit representations of the coboundary operator and the filtration order of the simplices, avoiding the explicit construction and storage of the filtration coboundary matrix.
Moreover, it makes use of apparent pairs, a simple but powerful method for constructing a discrete gradient field from a total order on the simplices of a simplicial complex, which is also of independent interest.
Our implementation shows substantial improvements over previous software both in time and memory usage.
\end{abstract}

\section{Introduction}
Persistent homology %
is a central tool in computational topology and topological data analysis.
It captures topological features of a \emph{filtration}, a growing one-parameter family of topological spaces, and tracks the lifespan of those features throughout the parameter range in the form of a collection of intervals called the \emph{persistence barcode}.
One of the most common constructions for a filtration from a geometric data set is the \emph{Vietoris–Rips} complex, which is constructed from a finite metric space by connecting any subset of the points with diameter bounded by a specified threshold with a simplex.

The computation of persistent homology has attracted strong interest in recent years \cite{MR3003900,MR2919613}, with at least 15 different implementations publicly available to date \cite{Plex,JPlex,JavaPlex,Dionysus,dionysus2,
PHom,Perseus,Bauer2017Phat,Bauer2014Distributed,GUDHI,CTL,Binchi2014JHoles,Libstick,Henselman2016Matroid,Zhang2019Hypha,zhang_et_al:LIPIcs:2020:12228,Cufar2020}.
Over the years, dramatic improvements in performance have been achieved, as demonstrated in recent benchmarks \cite{Otter2017Roadmap}.

The predominant approach to persistence computation consist of two steps: the construction of a filtration boundary matrix, and the computation of persistence barcodes using a matrix reduction algorithm similar to Gaussian elimination, which provides a decomposition of the filtered chain complex into indecomposable summands \cite{MR1310596}.
Among the fastest codes for the matrix reduction step is PHAT~\cite{Bauer2017Phat}, which has been created with the goal of assessing and understanding the relation and interplay of the various optimizations proposed in the previous literature on the matrix reduction algorithm.
In the course of that project, it became evident that often the construction of the filtration boundary matrix becomes the bottleneck for the computation of Vietoris–Rips barcodes.

The approach followed in Ripser \cite{Ripser} is to avoid the construction and storage of the filtration boundary matrix as a whole, discarding and recomputing parts of it when necessary.
In particular, 
instead of representing the coboundary map explicitly by a matrix data structure, it is given only algorithmically, recomputing the coboundary of a simplex whenever needed.
The filtration itself is also not specified explicitly but only algorithmically, via a method for comparing simplices with respect to their appearance in the filtration order, together with a method for computing the cofacets of a given simplex and their diameters.
The initial motivation for pursuing this strategy was purely to reduce the memory usage, possibly at the expense of an increased running time.
Perhaps surprisingly, however, this approach also turned out to be substantially faster than accessing the coboundary from memory.
This effect can be explained by the fact that, on current computer architectures, memory access is much more expensive than elementary arithmetic operations.

The computation of persistent homology as implemented in Ripser involves four key optimizations to the matrix reduction algorithm, two of which have been proposed in the literature before.
While our implementation is specific to Vietoris--Rips filtrations, the ideas are also applicable to persistence computations for other filtrations as well.

\paragraph{Clearing birth columns}
The standard matrix reduction algorithm does not make use of the special structure of a boundary matrix $D$, which satisfies $D^2=0$, i.e., boundaries are always cycles.
Ignoring this structure leads to a large number of unnecessary and expensive matrix operations in the matrix reduction, computing a large number of cycles that are not used subsequently.
The \emph{clearing} optimization (also called \emph{twist}), suggested by \citet{Chen2011Persistent}, avoids the computation of those cycles.

\paragraph{Cohomology}
The use of cohomology for persistence computation was first suggested by \citet{MR2854319}.
The authors establish certain dualities between persistent homology and cohomology and between absolute and relative persistent cohomology.
As a consequence, the computation of persistence barcodes can also be achieved as a cohomology computation.
A surprising observation, resulting from an application of persistent cohomology to the computation of circular coordinates by \citet{MR2787567}, was that the computation of persistent cohomology is often much faster than persistent homology.
This effect has been subsequently confirmed by \citet{Bauer2017Phat}, who further observed that the obtained speedup also depends heavily on the use of the clearing optimization proposed by \citet{Chen2011Persistent}, which is also employed implicitly in the cohomology algorithm of \cite{MR2854319}.
Especially for Vietoris--Rips filtrations and low homological degree, a decisive speedup is obtained, but only when both cohomology and clearing are used in conjunction.
A fully satisfactory explanation of this phenomenon has not been given previously in the literature.
In the present paper, we provide a simple counting argument that sheds light on this computational asymmetry between persistent homology and cohomology of Rips filtrations.

\paragraph{Implicit matrix reduction}
The computation of persistent homology usually relies on an explicit construction of a filtration coboundary matrix, which is then transformed to a reduced form, from which the persistence barcode can be read off directly.
In contrast, our approach is to decouple the description of the filtration and of the boundary operator, representing both the filtration and the coboundary matrix only algorithmically instead of explicitly.
Specifically, using a fixed lexicographic order for the $k$-simplices, independent of the filtration, 
the boundary and coboundary matrices for a full simplex on $n$ vertices are completely determined by the dimension $k$ and the number $n$, and their columns can simply be recomputed instead of being stored in memory.
Likewise, the filtration order of the simplices is defined to depend only on the distance matrix together with a fixed choice of total order on the simplices, used to break ties when two simplices appear simultaneously in the filtration.
Together, the filtration and the boundary map can be encoded using much less information than storing the coboundary matrix explicitly.
The algorithmic representation of the coboundary matrix in Ripser loosely resembles the use of lazy evaluation in the infinite-dimensional linear algebra framework of \citet{10.1109/HPTCDL.2014.10}.

Furthermore, we also avoid the storage of the reduced matrix as a whole, retaining only the much smaller \emph{reduction matrix}, which encodes the column operations applied to the coboundary matrix.
Besides the current column of the reduced matrix on which operations are performed, only information about the pivots of the reduced matrix is stored in memory.
In addition, only the pivots which can not be obtained directly from the unreduced matrix are stored in memory, as explained next.
The implicit representation of the reduced matrix by a reduction matrix has also been used in the cohomology algorithm by \citet{MR2854319}, which is implemented in \cite{GUDHI,Dionysus}.
In contrast to our implementation, however, in those implementations the unreduced filtration coboundary matrix is still stored explicitly.

\paragraph{Apparent and emergent pairs}
Further improvements to persistence computation can be obtained by expliting a certain easily identified type of persistence pair, called an \emph{apparent pair}, is encountered.
The pairing of a given simplex in an apparent pair can be determined by a purely local condition, depending only on the facets and cofacets of the simplex, and thus can be read off the filtration (co)boundary matrix directly without any matrix reduction.
In addition, since an apparent pair determines pivots in the boundary and the coboundary matrix, those pivots can be recomputed quickly, and thus only the pivots not corresponding to apparent pairs have to be stored in memory for the implicit matrix reduction algorithm.

Generalizing the notion of apparent pairs, \emph{emergent pairs} are persistence pairs that become apparent during the reduction, and can be read off the partially reduced matrix directly.
The construction of the filtration coboundary matrix columns can be cut short when an apparent or emergent pair is encountered.
During the enumeration of cofacets of a simplex for an appropriate refinement of the original filtration,
apparent and emergent pairs of persistence $0$ can be readily identified, circumventing the construction of the full coboundary of the simplex.
Since a large portion of all pairs appearing in the computation arises this way, it becomes unnecessary to construct the entire filtration (co)boundary matrix, and the speedup obtained from this shortcut is substantial.
Apparent pairs also provide a simple and natural construction for a discrete gradient (in the sense of discrete Morse theory) from a simplexwise filtration.

We note that special cases of the apparent pairs construction have been described in the literature before, and several equivalent variants have appeared in the literature, seemingly independently from the present work, after the public release of Ripser.
In particular, \citet{Kahle2011Random} described the construction of a discrete gradient on a simplicial complex based on a total order of the vertices, which is used to derive bounds on the topological complexity of random Vietoris--Rips complexes above the thermodynamic limit.
Indeed, our definition of apparent pairs arose from the goal of generalizing Kahle's construction to general filtrations of simplicial complexes.
We verify in \cref{lem: Kahle Gradient} that the discrete gradient constructed in that paper coincides with the apparent pairs of a simplexwise filtration given by the lexicographic order on the simplices.
Apparent pairs were also considered by \citet{Delgado2015Skeletonization}
as \emph{close pairs} in the context of cancelation of critical points in discrete Morse functions.
More recently, apparent pairs have been described by \citet[Remark 8.4.2]{HenselmanPetrusek2017Matroids} as \emph{minimal pairs of a linear order} and employed in the software Eirene \cite{Henselman2016Matroid}, which has been developed simultaneously and independently of Ripser.
In Eirene, apparent pairs are used to improve the performance of persistence computations by avoiding the construction of large parts of the boundary matrix, similar to the use of apparent and emergent pairs in Ripser.
An elaborate focus lies on the choice of refinement of the Vietoris--Rips filtration, aiming for a large number of pairs.
Following the first release of Ripser in 2016, apparent pairs have been studied by \citet{lampret2020chain} in the more general context of algebraic discrete Morse theory under the name \emph{steepness pairing}.
In a computational context, they have also been employed for parallel and multi-scale (coarse-to fine) persistence computation on the GPU by \citet[Definition 4]{Mendoza2017Parallel}, and in hybrid GPU/CPU variantd of PHAT and Ripser developed by \citet{Zhang2019Hypha,zhang_et_al:LIPIcs:2020:12228}.
Furthermore, a reimplementation of Ripser in Julia has been developed by \citet{Cufar2020}, and a lockfree shared-memory adaptation of Ripser has been created by \citet{10.1145/3350755.3400244}.

While apparent pairs have not been considered a central part of discrete Morse theory and of persistent homology so far, we consider their importance in recent research and their multiple discovery as strong evidence that they will play a significant role in the further development of these theories.

\section{Preliminaries}
\label{sec: Background}

\paragraph{Simplicial complexes and filtrations}

Given a finite set $X$, an (abstract) \emph{simplex} on $X$ is simply a nonempty subset $\sigma \subseteq X$.
The \emph{dimension} of $\sigma$ is one less than its cardinality, $\dim \sigma = |\sigma| - 1$.
Given two simplices $\sigma \subseteq \tau$, we say that $\sigma$ is a \emph{face} of $\tau$, and that $\tau$ is a \emph{coface} of $\sigma$.
If additionally $\dim \sigma + 1 = \dim \tau$, we say that $\sigma$ is a \emph{facet} of $\tau$ (a face of codimension $1$), and that $\tau$ is a \emph{cofacet} of $\sigma$.

A finite (abstract) \emph{simplicial complex} is a collection $K$ of simplices $X$ that is closed under the face relation: if $\tau \in K$ and $\sigma \subseteq \tau$, then $\sigma \in K$.
The set $X$ is called the \emph{vertices} of $K$, and the subsets in $K$ are called \emph{simplices}.
A \emph{subcomplex} of $K$ is a subset $L \subseteq K$ that is itself a simplicial complex.

Given a finite simplicial complex $K$, a \emph{filtration} of $K$ is a collection of subcomplexes $(K_i)_{i \in I}$ of $K$, where $I$ is a totally ordered indexing set, such that $i \leq j$ implies $K_i \subseteq K_j$.
In particular, for a finite metric space $(X,d)$,  represented by a symmetric distance matrix, the \emph{Vietoris--Rips complex} at scale $t \in \mathbb R$ is the abstract simplicial complex
\[\Rips_t(X) = \{\emptyset \neq S \subseteq X \mid \diam S \leq t \}.\]
Vietoris--Rips complex were first introduced by \citet{MR1512371} as a means of defining a homology theory for general compact metric spaces, and later used by Rips in the study of hyperbolic groups (see \cite{MR919829}).
Their usage in topological data analysis was pioneered by \citet{SPBG:SPBG04:157-166}, foreshadowed by results of \citet{MR1368659} and \citet{MR1879057} on sampling conditions for recovering the homotopy type of a Riemannian manifold from a Vietoris--Rips complex.
Letting the scale parameter $t$ vary, the resulting filtration, indexed by $I=\mathbb R$, is a filtration of the full simplex $\Delta(X)$, called the \emph{Vietoris--Rips filtration}.
For this paper, other relevant indexing sets besides the real numbers $\mathbb R$ are the set of distances $\{d(x,y) \mid x, y \in X\}$ in a finite metric space $(X,d)$, and the set of simplices of $\Delta(X)$ equipped with an appropriate total order refining the order by simplex diameter, as explained later.

We call a filtration \emph{essential} if $i \neq j$ implies $K_i \neq K_j$.
A \emph{simplexwise filtration} of $K$ is a filtration such that for all $i \in I$ with $K_i \neq \emptyset$, there is some simplex $\sigma_i \in K$ and some index $j < i \in I$ such that $K_i \setminus K_j = \{\sigma_i\}$.
In an essential simplexwise filtration, the index $j$ is the predecessor of $i$ in $I$.
Thus, essential simplexwise filtrations correspond bijectively to total orders extending the face poset of $K$, up to isomorphism of the indexing set $I$.
In particular, in this case we often identify the indexing set with the set of simplices.
If a simplex $\sigma$ appears earlier in the filtration than another simplex~$\tau$, i.e., $\sigma \in K_i$ whenever $\tau \in K_i$, we say that $\tau$ is \emph{younger} than $\sigma$, and $\sigma$ is \emph{older} than $\tau$.

It is often convenient to think of a simplicial filtration as a diagram $K_\bullet \colon I \to \cat{Simp}$ of simplicial complexes indexed over some finite totally ordered set $I$,
such that all maps $K_i \to K_j$ in the diagram (with $i \leq j$) are inclusions.
In terms of category theory, $K_\bullet$ is a functor.

\paragraph{Reindexing and refinement of filtrations}

A \emph{reindexing} of a filtration $F_\bullet \colon R \to \cat{Simp}$ indexed over some totally ordered set $R$ is another filtration $K_\bullet \colon I \to \cat{Simp}$ such that $F_t = K_{r(t)}$ for some monotonic map $r \colon R \to I$, called \emph{reindexing map}.
If there is a complex $K_i$ that does not occur in the filtration $F_\bullet$, we say that $K_\bullet$ \emph{refines} $F_\bullet$.

As an example, the filtration $\Rips_\bullet(X)$ is indexed by the real numbers~$\mathbb R$, but can be condensed to an essential filtration $K_\bullet$, %
indexed by the finite set of pairwise distances of $X$.
In order to compute persistent homology, one needs to apply one further step of reindexing, refining the filtration to an essential simplexwise one, as described in detail later.

\paragraph{Sublevel sets of functions}
A function $f \colon K \to \mathbb R$ on a simplicial complex $K$ is \emph{monotonic} if $\sigma \subseteq \tau \in K$ implies $f(\sigma) \leq f(\tau)$.
For any $t \in \mathbb R$, the \emph{sublevel set} $f^{-1}(-\infty,t]$ of a monotonic function $f$ is a subcomplex.
The sublevel sets form a filtration of $K$ indexed over $\mathbb R$.
Clearly, any finite filtration $K_\bullet \colon I \to \cat{Simp}$ of simplicial complexes can be obtained as a reduction of some sublevel set filtration.
In particular, the Vietoris--Rips filtration is simply the sublevel set filtration of the diameter function.

Discrete Morse theory \cite{Forman1998Morse} studies the topology of sublevel sets for generic functions on simplicial complexes.
A \emph{discrete vector field} on a simplicial complex $K$ is a partition $V$ of $K$
into singleton sets and pairs $\{\sigma, \tau\}$ in which $\sigma$
is a facet of~$\tau$.
We call such a pair a \emph{facet pair}.
A monotonic function $f \colon K \to \mathbb R$ is a \emph{discrete Morse function} if the facet pairs $\{\sigma,\tau\}$ with $f(\sigma)=f(\tau)$ generate a discrete vector field~$V$, which is then called the \emph{discrete gradient} of $f$.
A simplex that is not contained in any pair of $V$ is called a
\emph{critical simplex}, and the corresponding value is a
\emph{critical value} of~$f$.

\paragraph{Persistent homology}
In this paper, we only consider simplicial homology with coefficients in a prime field $\mathbb F_p$, and write $H_*(K)$ as a shortcut for $H_*(K; \mathbb F_p)$.
Applying homology
to a filtration of finite simplicial complexes $K_\bullet \colon I \to \cat{Simp}$ yields another diagram $H_*(K_\bullet) \colon I \to \cat{Vect}_p$ of finite dimensional vector spaces over $\mathbb F_p$, often called a \emph{persistence module} \cite{Chazal2016Structure}.

If all vector spaces have finite dimension, 
such diagrams have a particularly simple structure: they decompose into a direct sum of \emph{interval} persistence modules, consisting of copies of the field $\mathbb F_p$ connected by the identity map over an interval range of indices, and the trivial vector space outside the interval \cite{Zomorodian2005Computing,CrawleyBoevey2015Decomposition}.
This decomposition is unique up to isomorphism, and the collection of intervals describing the structure, the \emph{persistence barcode}, is therefore a complete invariant of the isomorphism type, capturing the homology at each index of the filtration together with the maps connecting any two different indices.
In fact, a corresponding decomposition exists already on the level of filtered chain complexes \cite{MR1310596}, and this decomposition is constructed by algorithms for computing persistence barcodes.

If $K_\bullet$ is an essential filtration and $[i,j) \subseteq I$ is an interval in the persistence barcode of $K_\bullet$, then we call $i$ a \emph{birth index}, $j$ a \emph{death index}, and the pair $(i,j)$ an \emph{index persistence pair}.
Moreover, if $[i,\infty)$ is an interval in the persistence barcode of $K$, we say that $i$ is an \emph{essential (birth) index}.
For an essential simplexwise filtration $K_\bullet$, the indies $I$ are in bijection with the simplices, and so in this context we also speak about \emph{birth, death}, and \emph{essential simplices}, and we consider pairs of simplices as persistence pairs.
If $K_\bullet$ is a reindexing of a sublevel set filtration for a monotonic function $f$, we say that the pair $(\sigma_i,\sigma_j)$ has \emph{persistence} $f(\sigma_j)-f(\sigma_i)$.

\paragraph{Persistence computation using simplexwise refinement}
A reindexing $K_\bullet$ of a filtration $F_\bullet = K_\bullet \circ r$ can be used to obtain the persistent homology of $F_\bullet$ from that of $K_\bullet$ as \[H_*(F_\bullet) = H_*(K_\bullet \circ r) = H_*(K_\bullet) \circ r.\]
Note that this is a direct consequence of the fact that the two filtrations $F_\bullet,K_\bullet$, the reindexing map~$r$, and homology~$H_*$ are functors, and composition of functors is associative.

If the reindexing map is not surjective, the persistence barcode of the reindexed filtration $K_\bullet$ may contain intervals that do not correspond to intervals in the barcode of $F_\bullet$.
The preimage $r^{-1}[i,j) \subseteq R$ of an interval $[i,j) \subseteq I$ in the persistence barcode of $K_\bullet$ is then either empty, in which case we call $(i,j)$ a \emph{zero} persistence pair;
if $F_\bullet$ is the sublevel set filtration of $f$, this is the case if and only if $f(\sigma_j)=f(\sigma_i)$.
Otherwise, $r^{-1}[i,j)$ is an interval of the persistence barcode for $F_\bullet$, and all such intervals arise this way.
We summarize:

\begin{proposition}
Let $f \colon K \to \mathbb R$ be a monotonic function on a simplicial complex $K$, and let $K_\bullet \colon I \to \cat{Simp}$ be an essential simplexwise refinement of the sublevel set filtration $F_\bullet = f^{-1}(-\infty,\bullet]$, with $K_i = \{\sigma_k \mid k \in I, \, k \leq i\}$.
The~persistence barcode of $K_\bullet$ determines the persistence barcode of $F_\bullet$,
\[B(H_*(F_\bullet)) = %
\left\{r^{-1}[i,j) \neq \emptyset \mid [i,j) \in B(H_*(K_\bullet)) \right\},\]
with $r^{-1}[i,j) = [f(\sigma_i),f(\sigma_j))$ and $r^{-1}[i,\infty) = [f(\sigma_i),\infty)$.
\end{proposition}

\paragraph{Filtration boundary matrices}
Given a simplicial complex $K$ with a totally ordered set of vertices $X$, there is a canonical basis of the simplicial chain complex $C_*(K)$, consisting of the simplices oriented according to the specified total order.
A simplexwise filtration turns this into an ordered basis and gives rise to a \emph{filtration boundary matrix}, which is the matrix of the boundary operator of the chain complex $C_*(K)$ with respect to that ordered basis.
We may consider boundary matrices both for the combined boundary map $\partial_* \colon C_* \to C_*$ as well as for the individual boundary maps $\partial_d \colon C_d \to C_{d-1}$ in each dimension~$d$.
Generalizing the latter case, we say that a matrix $D$ with column indices $I_d \subset I$ and row indices $I_{d-1} \subset I$ is a \emph{filtration $d$-boundary matrix} for a simplexwise filtration $K_\bullet \colon I \to \cat{Simp}$ if for each $i \in I$, the columns of $D$ with indices $\leq i$ form a generating set of the $(d-1)$-boundaries $B_{d-1}(K_i)$.
This allows us to remove columns from a boundary matrix that are linear combinations of the previous columns, a strategy called \emph{clearing} that is discussed in \cref{sec: Clearing columns}.

\paragraph{Indexing simplices in the combinatorial number system}
\label{sec: combinatorial numbering system}
We now describe the \emph{combinatorial number system} \cite{Pascal1887,Knuth2011Generating}, which provides a way of indexing the simplices of the full simplex $\Delta(X)$ and of the Vietoris--Rips filtration $\Rips_\bullet(X)$ by natural numbers, and which has previously been employed for persistence computation in \cite{Bauer2014Distributed}.
Again, we assume a total order on the vertices $X=\{v_0,\dots,v_{n-1}\}$ of the filtration.
Using this order, we identify each $d$-simplex $\sigma$ with the sorted $(d+1)$-tuple of its vertex indices $(i_{d},\dots,i_0)$ in decreasing order $i_d > \dots > i_0$.
This induces a lexicographic order on the set of $d$-simplices, which we refer to as the \emph{colexicographic vertex order}.
The \emph{combinatorial number system} of order $d+1$ is the order-preserving bijection
\[(i_d,\dots,i_0) \mapsto \sum_{l=0}^{d}{i_l \choose {l+1}}\]
mapping the lexicographically ordered set of decreasing $(d+1)$-tuples of natural numbers
to the set of natural numbers $\{0,\dots,\binom{n}{d+1} - 1\}$,
as illustrated in the following value table for $d=2$.
\[
\begin{array}{c|c|c|c|c|c|c}
(2,1,0)&(3,1,0)&(3,2,0)&(3,2,1)&(4,1,0)&\dots& (n-3,n-2,n-1)\\
\hline
0 & 1 & 2 & 3 & 4 & \cdots & \binom{n}{d+1} - 1
\end{array}
\]
Note that for $k > n$ the convention ${n \choose k} = 0$ is used here.
As an example, the simplex $\{v_5,v_3,v_0\}$ is assigned the number 
\[\textstyle(5,3,0) \mapsto {5 \choose 3} + {3 \choose 2} + {0 \choose 1} = 10 + 3 + 0 = 13.\]
Conversely, if a $d$-simplex $\sigma$ with vertex indices $(i_d,\dots,i_0)$ has index $N$ in the combinatorial number system, the vertices of $\sigma$ can be obtained by a binary search, as described in \cref{sec:vertices_simplices}.

\paragraph{Lexicographic refinement of the Vietoris–Rips filtration}
\label{sec: lexicographic refinement}
We now describe an essential simplexwise refinement of the Vietoris--Rips filtration, as required for the computation of persistent homology.
To this end, we consider another lexicographic order on the simplices of the full simplex $\Delta(X)$ with vertex set $X$, 
given by ordering the simplices
\begin{itemize}
\item by diameter, 
\item then by dimension, 
\item then by reverse colexicographic vertex order.
\end{itemize}
We will refer to the simplexwise filtration resulting from this total order as the \emph{lexicographically refined Vietoris--Rips filtration}.
The choice of the reverse colexicographic vertex order has algorithmic advantages, explained in \cref{sec: computing cofacets}.

As an example, consider the point set 
$X = \{v_0=(0,0), v_1=(3,0), v_2=(0,4), v_3=(3,4)\} \subseteq \mathbb R^2$,
consisting of the vertices of a $3\times 4$ rectangle
with the Euclidean distance. 
We obtain the distance matrix
\[
\begin{pNiceMatrix}[first-row,last-col]
v_0 & v_1 & v_2 & v_3 &\\
0 & 3 & 4 & 5 & ~~v_0\\
3 & 0 & 5 & 4 & ~~v_1\\
4 & 5 & 0 & 3 & ~~v_2\\
5 & 4 & 3 & 0 & ~~v_3
\end{pNiceMatrix}
\]
and the table of simplices (top row) with their diameters (bottom row)
\[\arraycolsep=0.6ex%
\small
\begin{array}
{c|c|c|c|c|c|c|c|c|c|c|c|c|c|c}
(3) & (2) & (1) & (0) & (3,2) & (1,0) & (3,1) & (2,0) & (3,0) & (2,1) & (3,2,1) & (3,2,0) & (3,1,0) & (2,1,0) & (3,2,1,0)\\
\hline
0 & 0 & 0 & 0& 
3 & 3 & 4 & 4 & 5 & 5 & 
5 & 5 & 5 & 5 & 
5
\end{array}
\]
listed in order of the lexicographically refined Vietoris--Rips filtration.

\section{Computation}

In this section, we explain the algorithm for computing persistent homology implemented in Ripser, and discuss the various optimization employed to achieve an efficient implementation.

\subsection{Matrix reduction}

The prevalent approach to computing persistent homology is by column reduction \cite{CohenSteiner2006Vines} of the filtration boundary matrix.
We write $M_i$ to denote the $i$th column of a matrix $M$.
The \emph{pivot index} of $M_i$, denoted by $\colpivot M i$, is the largest row index of any nonzero entry, taken to be $0$ if all entries of $v$ are~$0$.
Otherwise, the corresponding nonzero entry is called the \emph{pivot entry}, denoted by $\colpivotentry M i$.
We define $\pivots M =\bigcup_i\colpivot M i \setminus \{0\}$.

A column $M_i$ is called \emph{reduced} if $\colpivot M i$ cannot be decreased using column additions by scalar multiples of columns $M_j$ with $j<i$.
Equivalently, $\colpivot M i$ is minimal among all pivot indices of linear combinations
\[
\sum_{j\leq i}\lambda_j M_j
\]
with $\lambda_i\neq 0$,
meaning that multiplication from the right by a regular upper triangular matrix $U$ leaves the pivot index of the column unchanged: $\colpivot M i = \colpivot {(MU)} i$.
In particular, a column $M_i$ is reduced if either $M_i = 0$ or all columns $M_j$ with $j<i$ are reduced and satisfy $\colpivot M j \neq \colpivot M i$.
A matrix~$M$ is called \emph{reduced} if all of its columns are reduced.
The following proposition forms the basis of matrix reduction algorithms for computing persistent homology.

\begin{proposition}[\citet{CohenSteiner2006Vines}]
\label{prop: persistence pairs}
Let $D$ be a filtration boundary matrix, and let $V$ be a full rank upper triangular matrix such that $R=D\cdot V$ is reduced.
Then the index persistence pairs are 
\[\{(i,j) \mid i=\colpivot R j \neq 0 \}
,\] 
and the essential indices are 
\[
\{i \mid R_i=0,  i\not\in\pivots R\}.\]
\end{proposition}

A basis for the filtered chain complex that is compatible with both the filtration and the boundary maps is given by the chains
\[
\{R_j \mid j \text{ is a death index} \}
\cup
\{V_j \mid j \text{ is a death index} \}
\cup
\{V_i \mid i \text{ is an essential index} \}
,
\]
determining a direct sum decomposition of $C_*(K)$ into elementary chain complexes of the form
\[\dots \to 0 \to \langle V_j \rangle \stackrel\partial\to \langle R_j \rangle \to 0 \to \dots\]
for each death index $j$ and
\[\dots \to 0 \to \langle V_i \rangle \to 0 \to \dots\]
for each essential index $i$.
Taking intersections with the filtration $C_*(K_\bullet)$, we obtain elementary filtered chain complexes, in which $R_j$ is a cycle appearing in the filtration at index $i=\colpivot R j$ and becoming a boundary when $V_j$ enters the filtration at index $j$, and in which an essential cycle $V_i$ enters the filtration at index $i$.
The persistent homology is thus generated by the representative cycles
\[\{R_j \mid j \text{ is a death index} \}
\cup
\{V_i \mid i \text{ is an essential index} \},\]
in the sense that, for all indices $k \in I$, the homology $H_*(K_k)$ has a basis generated by the cycles
\[\{R_j \mid (i,j) \text{ is an index persistence pair with } k \in [i,j)  \}
\cup
\{V_i \mid i \text{ is an essential index with } k \in [i,\infty)\},\]
and for all pairs of indices $k,l \in I$ with $k \leq l$, the image of the map in homology $H_*(K_k) \to H_*(K_l)$ induced by inclusion has a basis generated by the cycles
\[\{R_j \mid (i,j) \text{ is an index persistence pair with } k,l \in [i,j)  \}
\cup
\{V_i \mid i \text{ is an essential index with } k \in [i,\infty)\}.\]
An algorithm for computing the matrix reduction $R = D \cdot V$ is given below as \cref{Matrix reduction}.
It can be applied either to the entire filtration boundary matrix in order to compute persistence in all dimensions at once, or to the filtration $d$-boundary matrix, resulting in the persistence pairs of dimensions $(d-1,d)$ and the essential indices of dimension~$d$.
This algorithm appeared for the first time in \cite{CohenSteiner2006Vines}, rephrasing the original algorithm for persistent homology \cite{Edelsbrunner2002Topological} as a matrix algorithm.
An algorithmic advantage to the decomposition algorithm described in \cite{MR1310596} is that it does not require any row operations. 

Note that the column operations involving the matrix $V$ are often omitted if the goal is to compute just the persistence pairs and representative cycles are not required.
In \cref{sec: Implicit filtration}, we will use the matrix~$V$ nevertheless to implicitly represent the matrix $R = D \cdot V$.

\begin{algorithm}[h]
\caption{Matrix reduction and persistence pairs}
\label{Matrix reduction}
\begin{algorithmic}%
\Require \\
$D$: $I \times J$ filtration boundary matrix (with row indices $I$ and column indices $J$)
\Ensure \\
$V$: full rank upper triangular $J \times J$ matrix,
$R = D \cdot V$: reduced matrix, \\
$P$: persistence pairs, 
$E$: essential indices

\smallskip
\hrule
\smallskip

\State{$P:=\emptyset$}
\For{$j \in J$ in increasing order}
\State{$R_j := D_j$}
\State{$V_j := e_j$}
\While{there exist $k < j$
       with $\colpivot R k = \colpivot R j$}
\State{$\lambda := \colpivotentry R j / \colpivotentry R k$ }
\State{$R_j := R_j - \lambda \cdot R_k$}
\Comment{For implicit matrix reduction (\cref{sec: Implicit filtration})}: $R_j := R_j - \lambda \cdot D \cdot V_k$ 
\State{$V_j := V_j - \lambda \cdot V_k$ }

\EndWhile

\If{$(i:=\colpivot R j) \neq 0$}
\State{append $(i,j)$ to $P$}
\Else
\State{append $j$ to $E$}
\EndIf

\EndFor
\State \Return $V$, $R$, $P$, $E$
\end{algorithmic}
\end{algorithm}

Typically, reducing a column at a birth index tends to be significantly more expensive than one with a death index.
This observation can be explained using the time complexity analysis for the matrix reduction algorithm given in \cite[Section VII.2]{Edelsbrunner2010Computational}: the reduction of a column for a $d$-simplex with death index~$j$ and corresponding birth index~$i$ requires at most $(d+1)(j-i)^2$ steps, while the reduction of a column with birth index $i$ requires at most $(d+1)(i-1)^2$ steps.
Typically, the index persistence $(j-i)$ is quite small, while the reduction of birth columns indeed becomes expensive for large birth indices $i$.
The next subsection describes a way to circumvent these birth column reductions whenever possible.

\subsection{Clearing inessential birth columns}
\label{sec: Clearing columns}

An optimization to the matrix reduction algorithm, due to \citet{Chen2011Persistent}, is based on the observation that the computation of persistence pairs according to \cref{prop: persistence pairs} does not involve any column with an inessential birth index.
Reducing those columns to zero is therefore unnecessary, and avoiding their reduction can lead to dramatic improvements in running time.
The \emph{clearing} optimization simply sets a column $R_i$~to~$0$ if $i$ is the pivot index of another column, $i = \colpivot R j$.

As proposed in \cite{Chen2011Persistent}, clearing yields only the reduced matrix~$R$, and the method is described in the survey by \citet{Morozov2017Persistent} as incompatible with the computation of the reduction matrix $V$.
In fact, however, the clearing optimization can actually be extended to also obtain the reduction matrix, which plays a crucial role in our implementation.
To see this, consider an inessential birth index $i = \colpivot R j$, corresponding to a cleared column.
Obtaining the requisite full rank upper triangular reduction matrix~$V$ requires an appropriate column $V_i$ such that $D \cdot V_i = R_i = 0$, i.e., $V_i$ is a cycle.
It suffices to simply take $V_i = R_j$; by construction, this column is a boundary, and since $\colpivot V i = i$, the resulting matrix $V$ will be full rank upper triangular.

Note that when removing columns with birth indices from a filtration $d$-boundary matrix, the remaining columns remain a generating set for the boundaries as required by the definition.
The resulting method is summarized in \cref{Matrix reduction with clearing}.

\begin{algorithm}
\caption{Matrix reduction and persistence pairs with clearing}
\label{Matrix reduction with clearing}
\begin{algorithmic}%
\Require 
\\
$D$: $I \times J$ filtration $d$-boundary matrix  (with row indices $I$ and column indices $J$),\\
$\widetilde R$: reduced filtration $(d+1)$-boundary matrix,
$\widetilde P$: persistence pairs of dimensions $(d,d+1)$
\Ensure 
\\
$V$: full rank upper triangular $J \times J$ matrix, 
$R = D \cdot V$: reduced filtration $d$-boundary matrix, \\
$P$: persistence pairs of dimensions $(d-1,d)$, 
$E$: essential indices of dimension $d$

\smallskip
\hrule
\smallskip

\State{$\widehat J := J \setminus \pivots \widetilde R$} %

\State{$\widehat D := I \times \widehat J$ submatrix of $D$}

\State{Apply \cref{Matrix reduction}
to
reduce $\widehat D$ to $\widehat R = \widehat D \cdot \widehat V$
and obtain the persistence pairs $P$ and the essential indices $E$}

\State{Extend $\widehat R$ to a $I \times J$ matrix $R$ by filling in zeros}

\State{Extend $\widehat V$ to a $J \times J$ matrix $V$ by filling in zeros}

\For{$(i,j) \in \widetilde P$}
\State{$V_i := \widetilde R_j$}
\EndFor
\State \Return $V$, $R$, $P$, $E$
\end{algorithmic}
\end{algorithm}

The birth indices $i$ of persistence pairs $(i,j)$ are identified with the reduction of column $j$.
In order to ensure that this happens before the algorithm arrives at column $i$, so that the column is cleared already before it would get reduced, the matrices for the boundary maps $\partial_d$
are reduced in order of decreasing dimension $d = (p+1), \dots, 1$.
For each index persistence pair $(i,j)$ computed in the reduction of the boundary matrix for $\partial_d$, the corresponding column for index~$i$ can now be removed from the boundary matrix for $\partial_{d-1}$.
Note however that computing persistent homology in dimensions~$0\leq d \leq p$ still requires the reduction of the full boundary matrix $\partial_{p+1}$.
This can become very expensive, especially if there are many $(p+1)$-simplices, as in the case of a Vietoris–Rips filtration.
In this setting, the complex~$K$ is the $(p+1)$-skeleton of the full simplex on $n$~vertices.
The standard matrix reduction algorithm for persistent homology requires the reduction of one column per simplex of dimension $1 \leq d \leq p + 1$, amounting to
\[
\sum_{d=1}^{p+1} \underbrace{\binom{n}{d+1}}_{\dim C_{d}(K)} %
= \sum_{d=1}^{p+1} \underbrace{\binom{n-1}{d}}_{\dim B_{d-1}(K)} + \sum_{d=1}^{p+1} \underbrace{\binom{n-1}{d+1}}_{\dim Z_{d}(K)}
\]
columns in total.
Note that $\dim B_{d-1}(K)$ equals the number of death columns and $\dim Z_{d}(K)$ equals the number of birth columns in the $d$-boundary matrix.
As an example, for $p=2,\,n=192$ we obtain $56\,050\,096$
columns, of which $1\,161\,471$ are death columns and $54\,888\,625$ are birth columns.
Using the clearing optimization, this number is lowered to
\[
\sum_{d=1}^{p+1} \underbrace{\binom{n-1}{d}}_{\dim B_{d-1}(K)}
+ \underbrace{\binom{n-1}{p+2}}_{\dim Z_{p+1}(K)}
=
\sum_{d=1}^{p+2} \binom{n-1}{d}
=
\sum_{d=0}^{p+1} \binom{n-1}{d+1}
\]
columns; again, for $p=2,\,n=192$ this still yields $54\,888\,816$ columns, of which $1\,161\,471$ are death columns and $53\,727\,345$ are birth columns.
Because of the large number of birth columns arising from $(p+1)$-simplices, the use of clearing alone thus does not lead to a substantial improvement yet.

\subsection{Persistent Cohomology}
\label{Cohomology}
The clearing optimization can be used to a much greater effect by computing persistence barcodes using cohomology instead of homology of Vietoris--Rips filtrations.
As noted by \citet{MR2854319}, for a filtration $K_\bullet$ of a simplicial complex $K$ the persistence barcodes for homology $H_*(K_\bullet)$ and cohomology $H^*(K_\bullet)$ coincide, since for coefficients in a field, cohomology is a vector space dual to homology \cite{Munkres1984Elements}, and the barcode of persistent homology (and more generally, of any pointwise finite-dimensional persistence module) is uniquely determined by the ranks of the internal linear maps in the persistence module, which are preserved by vector space duality.

The filtration of chain complexes $C_*(K_\bullet)$ gives rise to a diagram of cochain complexes $C^d(K_\bullet)$, with reversed order on the indexing set.
Since cohomology is a contravariant functor, the morphisms in this  diagrams are however surjections instead of injections.
To obtain a setting that is suitable for our reduction algorithms, we instead consider the filtration of relative cochain complexes $C^d(K,K_\bullet)$.
The \emph{filtration coboundary matrix} for $\delta \colon C^d(K,K_\bullet) \to C^{d+1}(K,K_\bullet)$ is given as the transpose of the filtration boundary matrix with rows and columns ordered in reverse filtration order \cite{MR2854319}.
The persistence barcodes for relative cohomology $H^*(K,K_\bullet)$ uniquely determine those for absolute cohomology $H^*(K_\bullet)$ (and coincide with the respective homology barcodes by duality).
This correspondence can be seen as a consequence of the fact that the short exact sequence of cochain complexes of persistence modules
\[
0 \to C^*(K,K_\bullet) \to C^*(K) \to C^*(K_\bullet) \to 0
\]
(where $C^*(K)$ is interpreted as a complex of constant persistence modules) 
gives rise to a long exact sequence
\[
\cdots \to H^d(K) \to  H^d(K,K_{\bullet}) \stackrel{\delta^*}\to H^{d+1}(K_{\bullet}) \to H^{d+1}(K) \to \cdots 
\]
in cohomology, which can be seen to split at $H^d(K,K_{\bullet})$ and at $H^{d+1}(K_{\bullet})$, with $\op{im}\delta^*$ as the summand corresponding to the bounded intervals in either barcode \cite{bauer2020structure}. 
Now the persistence pairs $(j,i)$ of dimensions $(d,d-1)$ for relative cohomology $H^d(K,K_i)$ correspond to persistence pairs $(i,j)$ of dimensions $(d-1,d)$ for (absolute) homology $H_{d-1}(K_i)$ in one dimension below, i.e., a death index becomes a dual non-essential birth index and vice versa, while the essential birth indices for $H^d(K,K_i)$ remain essential birth indices for $H_d(K_i)$ in the same dimension.
Thus, the persistence barcode can also be computed by matrix reduction of the filtration coboundary matrix.
Note that this is equivalent to row reduction of the filtration boundary matrix, reducing the rows from bottom to top.

Since the coboundary map increases the degree, in order to apply the clearing optimization described in \cref{sec: Clearing columns}, the filtration $d$-coboundary matrices are now reduced in order of increasing dimension using \cref{Matrix reduction with clearing}.
This yields the relative persistence pairs of dimensions $(d+1,d)$ for $H^{d+1}(K,K_i)$, corresponding to the absolute persistence pairs of dimensions $(d,d+1)$ for $H_d(K_i)$, and the essential indices of dimension $d$.
This is the approach used in Ripser.

The crucial advantage of using cohomology to compute the Vietoris–Rips persistence barcodes in dimensions $0 \leq d \leq p$ lies in avoiding the expensive reduction of columns whose birth indices correspond to $(p+1)$-simplices, as discussed in \cref{sec: Clearing columns}.
To illustrate the difference, we first consider cohomology without clearing.
Note that for persistent cohomology, the number of column reductions performed by the standard matrix reduction (\cref{Matrix reduction}) is
\[\sum_{d=0}^{p} \underbrace{\binom{n}{d+1}}_{\dim C^{d}(K)} %
 = \sum_{d=0}^{p} \underbrace{\binom{n-1}{d+1}}_{\dim B^{d+1}(K)}
 + \sum_{d=0}^{p} \underbrace{\binom{n-1}{d}}_{\dim Z^{d}(K)}
;
\]
again, for $K$ the $(p+1)$-skeleton of the full simplex on $n$ vertices with $p=2,\,n=192$, this amounts to $1\,179\,808$ columns, of which $1\,161\,471$ are death columns and $18\,337$ are birth columns.
While this number is significantly smaller than for homology, for small values of $d$ the number of rows of the coboundary matrix, $\binom{n}{d+1}$, is much larger than that of the boundary matrix, $\binom{n}{d}$, and thus the reduction of birth columns becomes prohibitively expensive in practice.
Consequently, reducing the coboundary matrix without clearing has not been observed as more efficient in practice than reducing the boundary matrix \cite{Bauer2017Phat}.
However, in conjunction with the clearing optimization, only
\[
\sum_{d=0}^{p} \underbrace{\binom{n-1}{d+1}}_{\dim B^{d+1}(K)}
  + \underbrace{\binom{n-1}{0}}_{\dim Z^0(K)}
=
\sum_{d=0}^{p+1} \binom{n-1}{d}
\]
columns remain to be reduced; for
$p=2,\,n=192$ we get $1\,161\,472 = 1\,161\,471 + 1$ columns, of which $1\,161\,471$ are death columns and only one is a birth column, corresponding to the single essential class in dimension~$0$.
In addition, typically a large fraction of the death columns will be reduced already from the beginning, as observed in \cref{sec: Experiments}.
Thus, in practice, the combination of clearing and cohomology drastically reduces the number of columns operations in comparison to \cref{Matrix reduction}.

\subsection{Implicit matrix reduction}
\label{sec: Implicit filtration}

The matrix reduction algorithm can be modified slightly to yield a variant in which only the reduction matrix $V$ is represented explicitly in memory.
The columns of the coboundary matrix $D$ are computed on the fly instead, by a method that enumerates the nonzero entries in a given column of $D$ in reverse colexicographic vertex order of the corresponding rows.
Specifically, using the combinatorial number system to index the simplices on the vertex set $\{0,\dots,n-1\}$, the cofacets of a simplex can be enumerated efficiently, as described in more detail in \cref{sec: computing cofacets}.
The matrix $R = D \cdot V$ is now determined implicitly by $D$ and~$V$.
During the execution of the algorithm, only the current column $R_j$ on which additions are performed is kept in memory.
For all previously reduced columns $R_k$, with $k < j$, only the pivot index and its entry, $\colpivot R k$ and $\colpivotentry R k$, are stored in memory.
Whenever needed in the algorithm, those columns are recomputed on the fly as $R_k = D \cdot V_k$ (see \cref{Matrix reduction}).
Note that the extension of clearing to the reduction matrix described in \cref{sec: Clearing columns} is crucial for an efficient implementation of implicit matrix reduction.

To further decrease the memory usage, we apply another minor optimization.
Note that in the matrix reduction algorithm \cref{Matrix reduction}, only a death index $k$  ($R_k = D \cdot V_k \neq 0$) may satisfy the condition $\colpivot R k = \colpivot R j$ in \cref{Matrix reduction}.
Hence, only columns with a death index are used later in the computation to eliminate pivots.
Consequently, our implementation does not actually maintain the entire matrix $V$, but only stores the columns of $V$ with a death index.
In other words, it does not store explicit generating cocycles for persistent cohomology, but only their cobounding cochains.

\subsection{Apparent pairs}

\label{shortcut}

We now discuss a class of persistence pairs that can be identified in a Vietoris–Rips filtration directly from the boundary matrix without reduction, and often even without actually enumerating all cofacets, i.e., without entirely constructing the corresponding columns of the coboundary matrix.
The columns in question are already reduced in the coboundary matrix, and hence remain unaffected by the reduction algorithm.
Most relevant to us are persistence pairs that have persistence zero with respect to the original filtration parameter, meaning that they arise only as an artifact of the lexicographic refinement and do not contribute to the Vietoris–Rips barcode itself.
In practice, most of the pairs arising in the computation of Vietoris–Rips persistence are of this kind (see \cref{sec: Experiments}, and also \cite[Section 4.2]{Zhang2019Hypha}).

\begin{definition}
Consider a simplexwise filtration $K_\bullet$ of a finite simplicial complex $K$.
We call a pair of simplices $(\sigma,\tau)$ of $K$ an
\emph{apparent pair} of $K_\bullet$ if both
\begin{itemize}
\item $\sigma$ is the youngest facet of $\tau$, and
\item $\tau$ is the oldest cofacet of $\sigma$.
\end{itemize}
Equivalently, all entries in the filtration boundary matrix of $K_\bullet$ below or to the left of $(\sigma,\tau)$ are $0$.
\end{definition}
The notion applies to any simplexwise filtration $K_\bullet$, which may arise as a simplexwise refinement of some coarser filtration $F_\bullet$, such as our lexicographic refinement of the Vietoris–Rips filtration.
As an example, consider the Vietoris–Rips filtration on the vertices of a rectangle from \cref{sec: lexicographic refinement}.
The filtration boundary matrices are
\[\small\arraycolsep=0.5ex
\partial_1 =
\begin{pNiceMatrix}[first-row,last-col]
(3,2) & (1,0) & (3,1) & (2,0) & (3,0) &(2,1)\\
1 & 0 & 1 & 0 & 1 & 0 & ~~(3)\\
\mathbf 1 & 0 & 0 & 1 & 0 & 1 & ~~(2)\\
0 & 1 & 1 & 0 & 0 & 1 & ~~(1)\\
0 & \mathbf 1 & 0 & 1 & 1 & 0 & ~~(0)
\end{pNiceMatrix}
\qquad
\partial_2 =
\begin{pNiceMatrix}[first-row,last-col]
(3,2,1) &(3,2,0) & (3,1,0) & (2,1,0)\\
1 & 1 & 0 & 0 & ~~(3,2)\\
0 & 0 & 1 & 1 & ~~(1,0)\\
1 & 0 & 1 & 0 & ~~(3,1)\\
0 & 1 & 0 & 1 & ~~(2,0)\\
0 & \mathbf 1 & 1 & 0 & ~~(3,0)\\
\mathbf 1 & 0 & 0 & 1 & ~~(2,1)
\end{pNiceMatrix}
\qquad
\partial_3 =
\begin{pNiceMatrix}[first-row,last-col]
(3,2,1,0)\\
1& ~~(3,2,1)\\
1& ~~(3,2,0)\\
1& ~~(3,1,0)\\
\mathbf 1& ~~(2,1,0)
\end{pNiceMatrix}
\]
with bold entries corresponding to apparent pairs.
The apparent pairs thus are
\[((0),(1,0)), \, ((2),(3,2)); \, ((2,1),(3,2,1)),\, ((3,0),(3,2,0)); ((2,1,0),(3,2,1,0)).\]
Note that in this example, every zero persistence pair is an apparent pair.

Apparent pairs provide a connection between persistence and discrete Morse theory \cite{Forman1998Morse}.
The collection of all apparent pairs pairs of a simplexwise filtration~$K_\bullet$ will simultaneously constitute a subset of the persistence pairs (\cref{lem: apparent pairs are persistence pairs}) and form a discrete gradient in the sense of discrete Morse theory (\cref{lem: apparent pairs are Morse pairs}).

\paragraph{Apparent pairs as persistence pairs}
As an immediate consequence of the definition of an apparent pair, we get the following lemma.

\begin{lemma}\label{lem: apparent pairs are persistence pairs}
Any apparent pair of a simplexwise filtration is a persistence pair.
\end{lemma}

\begin{proof}
Since the entries in the filtration boundary matrix of $K_\bullet$ to the left of an apparent pair $(\sigma,\tau)$ are $0$, the index of $\sigma$ is the pivot of the column of $\tau$ in the filtration boundary matrix.
Thus, the column of $\tau$ in the boundary matrix is already reduced from the beginning, and \cref{prop: persistence pairs} yields the claim.
\end{proof}

\begin{remark}
Note that the property of being an apparent pair does not depend on the choice of the coefficient field.
Indeed, the above statement holds for any choice of coefficient field for (co)homology.
In that sense, the apparent pairs are \emph{universal persistence pairs}.

Specifically, given an apparent pair $(\sigma,\tau)$, the chain complex $\dots \to 0 \to \langle \tau \rangle \to \langle \partial\tau \rangle \to 0 \to \dots$ can easily be seen to yield an indecomposable summand of the filtered chain complex $C_*(K_\bullet; \mathbb Z)$ with integer coefficients (by intersecting the chain complexes), and to generate an interval summand of persistent homology (as a diagram of Abelian groups indexed by a totally ordered set), and thus any other coefficients, by the universal coefficient theorem.
The boundary $\partial\tau$ enters the filtration simultaneously with~$\sigma$, and the appearance of $\sigma$ and $\tau$ in the filtration determine the endpoints of the resulting interval.

Dually, the coboundary $\delta\sigma$ enters the filtration simultaneously with~$\tau$, and the cochain complex 
$\dots \to 0 \to \langle \sigma \rangle \to \langle \delta\sigma \rangle \to 0 \to \dots$ yields an indecomposable summand of $C^*(K_\bullet; \mathbb Z)$, generating an interval summand of $H^*(K_\bullet; \mathbb Z)$ corresponding to the one mentioned above.
\end{remark}

\paragraph{Apparent pairs as gradient pairs}
The next two lemmas relate apparent pairs to discrete Morse theory.
First, we show that apparent pairs are gradient pairs.
\begin{lemma}\label{lem: apparent pairs are Morse pairs}
The apparent pairs of a simplexwise filtration form a discrete gradient.
\end{lemma}

\begin{proof}

Let $(\sigma, \tau)$ be an apparent pair, with $\dim \sigma = d$.
By definition, $\tau$ is uniquely determined by $\sigma$, and so $\sigma$~cannot appear in another apparent pair $(\sigma, \psi)$ for any $(d+1)$-simplex $\psi \neq \tau$.
We show that $\sigma$ also does not appear in another apparent pair $(\phi, \sigma)$ for any $(d-1)$-simplex $\phi$.
To see this, 
note that there is another $d$-simplex $\rho \neq \sigma$ that is also a facet of $\tau$ and a cofacet of $\phi$.
Since $\sigma$ is assumed to be the youngest facet of $\tau$,
the simplex $\rho$ is older than~$\sigma$.
In particular, $\sigma$ is not the oldest cofacet of $\phi$, and so $(\phi, \sigma)$ is not an apparent pair.
We conclude that no simplex appears in more than one apparent pair, i.e., the apparent pairs define a discrete vector field.

To show that this discrete vector field is a gradient, let $\sigma_1, \dots, \sigma_m$ be the simplices of $K$ in filtration order, and consider 
the function 
\[f \colon
\sigma_j \mapsto 
\begin{cases}
i & \text{ if there is an apparent pair } (\sigma_i, \sigma_j),\\
j & \text{ otherwise.}
\end{cases}
\]
To verify that $f$ is a discrete Morse function, first note that $f(\sigma_k) \leq k$.
Now let $\sigma_i$ be a facet of~$\sigma_j$.
Then $i < j$.
If $(\sigma_i, \sigma_j)$ is not an apparent pair, we have $f(\sigma_i) \leq i < j = f(\sigma_j) $.
On the other hand, if $(\sigma_i, \sigma_j)$ is an apparent pair, then $\sigma_i$ is the youngest facet of $\sigma_j$, i.e., $k \leq i$ for every facet $\sigma_k$ of $\sigma_j$, and thus $f(\sigma_k) \leq k \leq i = f(\sigma_j)$, with equality holding if and only if $i = k$.
We conclude that $f$ is a discrete Morse function whose sublevel set filtration is refined by $K$ and whose gradient pairs are exactly the apparent pairs of the filtration.
\end{proof}

For the previous example of a simplexwise filtration obtained fron the vertices of a rectangle, the resulting discrete Morse function as constructed in the above proof is shown in the following table.
\[\arraycolsep=0.6ex%
\small
\begin{array}
{c|c|c|c|c|c|c|c|c|c|c|c|c|c|c}
(3) & (2) & (1) & (0) & (3,2) & (1,0) & (3,1) & (2,0) & (3,0) & (2,1) & (3,2,1) & (3,2,0) & (3,1,0) & (2,1,0) & (3,2,1,0)\\
\hline
0 & 1 & 2 & 3& 
1 & 0 & 6 & 7 & 8 & 9 & 
9 & 8 & 12 & 13 & 
13
\end{array}
\]

The following lemma is a partial converse of the above \cref{lem: apparent pairs are Morse pairs}, showing that the gradient pairs of a discrete Morse function form a subset of the apparent pairs.
\begin{lemma}
\label{lem: gradient pairs are zero persistence apparent pairs}
Let $f$ be a discrete Morse function, and let $K_\bullet$ be a simplexwise refinement of the sublevel set filtration $F_\bullet = K_\bullet \circ r$ for $f$.
Then the gradient pairs of $f$ are precisely the  zero persistence apparent pairs of $K_\bullet$.
\end{lemma}

\begin{proof}
Any $0$-persistence pair $(\sigma,\tau)$ satisfies $f(\sigma)=f(\tau)$ and thus, by definition of a discrete Morse function, forms a gradient pair of $f$.

Conversely, any gradient pair $(\sigma,\tau)$ of $f$ satisfies $f(\sigma)=f(\tau)$ and $f(\rho)<f(\tau)$ for any facet $\rho \neq \sigma$ of $\tau$, and similarly, $f(\upsilon)>f(\tau)$ for any cofacet $\rho \neq \tau$ of $\sigma$.
Thus, in any simplexwise refinement of the sublevel set filtration, $\sigma$ is the youngest facet of $\tau$, and $\tau$ is the oldest cofacet of $\sigma$.
This means that $(\sigma,\tau)$ is an apparent zero persistence pair.
\end{proof}

Note that there might be apparent pairs of nonzero persistence.
In particular, starting with a discrete Morse function $f$ and constructing another discrete Morse function $\tilde f$ from $f$ as in the proof of \cref{lem: apparent pairs are Morse pairs}, the gradient pairs of $f$ form a subset of the gradient pairs of $\tilde f$.
In particular, the nonzero persistence apparent pairs of $f$ will be gradient pairs of $\tilde f$, but not of $f$.

\begin{remark}
The notion of apparent pairs generalize to the setting of algebraic Morse theory \cite{MR2151775,MR2171225,MR2488864} in a straightforward way \cite{lampret2020chain}.
In this setting, one considers a finitely generated free chain complex
of free modules $C_{d}$ over some ring $R$, equipped with a fixed ordered basis $\Sigma_d$, which is assumed to induce a filtration by subcomplexes.
In this setting, the facet and cofacet relations are defined by the more general condition that two basis elements $\sigma \in \Sigma_d$ and $\tau \in \Sigma_{d+1}$ have a nonzero coefficient in the boundary matrix for $\partial_{d+1}$.
In addition, an algebraic apparent pair $(\sigma,\tau)$ is required to satisfy the additional condition that this boundary coefficient is a unit in $R$.
Similarly to the simplicial setting, in this case $\sigma$ cannot appear in another apparent pair $(\phi, \sigma)$ for any facet $\phi$ of $\sigma$.
To see this, note that again there has to be another basis element $\rho \neq \sigma \in \Sigma_d$ that is also a facet of $\tau$ and a cofacet of $\phi$; if $\sigma$ were the only such basis element, then the coefficient of $(\phi, \tau)$ in the matrix of $\partial_{q}\circ\partial_{q+1}$ would be nonzero, which is excluded by the chain complex property.
The above \cref{lem: apparent pairs are persistence pairs,lem: apparent pairs are Morse pairs,lem: gradient pairs are zero persistence apparent pairs} therefore hold in this generalized setting as well.
Note however that in this case, apparent pairs are no longer guaranteed to be universal persistence pairs, as the condition defining the apparent pairs now depends of the choice of coefficients.
\end{remark}

\paragraph{Lexicographic discrete gradients}
The construction proposed by \citet{Kahle2011Random} for a discrete gradient~$V_L$ on a simplicial complex $K$ from a total vertex order can be understood as a special case of the apparent pairs gradient.
The definition of the gradient $V_L$ is as follows.
Consider the vertices $v_1, \dots v_n$ of the simplicial complex $K$ in some fixed total order.
Whenever possible, pair a simplex $\sigma = \{v_{i_d}, \dots v_{i_1}\}$, $i_d > \dots > i_1$, with the simplex $\tau = \{v_{i_d}, \dots v_{i_1},v_{i_0}\}$ for which $i_0 < i_1$ is minimal.
These pairs $(\sigma,\tau)$ form a discrete gradient \cite{Kahle2011Random}, which we call the \emph{lexicographic gradient.}

We illustrate how this gradient $V_L$ can be considered as a special case of our definition of apparent pairs for the \emph{lexicographic filtration} of $K$, where the simplices are ordered by dimension, and simplices of the same dimension are ordered lexicographically according to the chosen vertex order.

\begin{lemma}
\label{lem: Kahle Gradient}
The lexicographic gradient $V_L$ is the apparent pairs gradient of the lexicographic filtration.
\end{lemma}
\begin{proof}
To see that any pair $(\sigma,\tau)$ in $V_L$ is an apparent pair, observe that $i_0$ is chosen such that $\tau$ is the lexicographically smallest cofacet of $\sigma$.
Moreover, $\sigma$ is clearly the lexicographically largest facet of $\tau$.

Conversely, assume that $(\sigma,\tau)$ is an apparent pair for the lexicographic filtration.
Let $\tau = \{v_{i_d}, \dots, v_{i_1},v_{i_0}\}$.
Then $\sigma = \{v_{i_d}, \dots, v_{i_1}\}$ is the lexicographically largest facet of $\tau$, and we have $i_0 < i_1$.
Moreover, since $\tau$ is the lexicographically smallest cofacet of $\sigma$, the index $i_0$ is minimal among all indices~$i$ such that $\{v_{i_d}, \dots, v_{i_1},v_{i_0}\}$ forms a simplex.
\end{proof}

\paragraph{Vietoris--Rips filtrations}
Having discussed the Morse-theoretic interpretation of apparent pairs, we now illustrate their relevance for the computation of persistence.
Focusing on the lexicographically refined Rips filtration, we first describe the apparent pair of persistence zero in a way that is suitable for computation.

\begin{proposition}
\label{prop: zero apparent pairs}
Let $\sigma, \tau$ be simplices in the lexicographically refined Rips filtration.
Then $(\sigma, \tau)$ is a zero persistence apparent pair if and only if
\begin{itemize}
\item $\tau$ is the lexicographically 
maximal cofacet of $\sigma$ such that $\diam(\tau) = \diam(\sigma)$, and
\item $\sigma$ is the lexicographically 
minimal facet of $\tau$ such that $\diam(\tau) = \diam(\sigma)$.
\end{itemize}
\end{proposition}
\begin{proof}
The two conditions hold for any zero persistence apparent pair by definition.
It remains to show that the two conditions also imply that $(\sigma, \tau)$ is an apparent pair.
Recall that in the lexicographically refined Rips filtration, simplices are sorted by diameter, then by dimension, and then in (reverse) lexicographic order.
The
simplex $\tau$
must be the oldest cofacet of $\sigma$ in the filtration order: no cofacet of~$\sigma$ can have a diameter less than $\diam\tau = \diam\sigma$, and among the cofacets of~$\sigma$ with the same diameter as~$\tau$, the lexicographically 
maximal cofacet~$\tau$ is the oldest one by construction of the lexicographic filtration order.
Similarly, %
$\sigma$~%
must be the oldest cofacet of $\sigma$ in the filtration order.
\end{proof}

As it turns out, a large portion of all persistence pairs arising in the persistence computation for Rips filtrations can be found this way, and the savings from not having to enumerate all cofacets are substantial.
The following theorem gives a partial explanation of this observation, for generic finite metric spaces and persistence in dimension~$1$.

To set the stage, note that in a a simplexwise refinement of a Vietoris--Rips filtration, any apparent pair, and more generally, any emergent facet pair $(\sigma,\tau)$ of dimensions $(k,k+1)$ with $k \geq 1$ necessarily has persistence zero, as the diameter of $\tau$ equals the maximal diameter of its facets, which by definition is the diameter of $\sigma$.
The following theorem establishes a partial converse to this fact in dimension $1$ and under a certain genericity assumption on the metric space.

\begin{theorem}
\label{zero pairs in dim 1 are apparent}
Let $(X,d)$ be a finite metric space with distinct pairwise distances,
and let $K_\bullet$ be a simplexwise refinement of the Vietoris--Rips filtration for $X$.
Among the persistence pairs of $K_\bullet$ in dimension $1$, the zero persistence pairs are precisely the apparent pairs.
\end{theorem}

\begin{proof}
As noted above, every apparent pair is a zero persistence pair.
Conversely, let $(\sigma,\tau)$ be a zero persistence pair of dimensions $(1,2)$.
Since the edge diameters are assumed to be distinct, the edge $\sigma$ must be the youngest facet of $\tau$ in $K_\bullet$.
Now let $\psi$ be the oldest cofacet of $\sigma$.
We then have $\diam(\sigma) \leq \diam(\psi) \leq \diam(\tau)$, and since $(\sigma,\tau)$ is a zero persistence pair, this implies $\diam(\sigma) = \diam(\psi) = \diam(\tau)$.
Since pairwise distances are assumed to be distinct, any edge $\rho \subset \psi$, $\rho \neq \sigma$ must satisfy $\diam(\rho) < \diam(\psi)$.
This implies that $\sigma$ has to be the youngest facet of $\psi$.
Hence, $(\sigma,\psi)$ is an apparent pair.
But since $(\sigma,\tau)$ is assumed to be a persistence pair, we conclude with \cref{lem: apparent pairs are persistence pairs} that $\psi = \tau$, and so $(\sigma,\tau)$ is an apparent pair.
\end{proof}

This result implies that every column in the filtration $1$-coboundary matrix that is not amenable to the emergent shortcut described in \cref{shortcut} actually corresponds to a proper interval in the Vietoris--Rips barcode.
In other words, the zero persistence pairs of the simplexwise refinement do not incur any computational cost in the matrix reduction.

\paragraph{Emergent persistence pairs}

The definition of apparent pairs generalizes to the persistence pairs in a simplexwise filtration that become apparent in the reduced matrix during the matrix reduction, and are therefore called \emph{emergent pairs}. They come in two flavors.

\begin{definition}
Consider a simplexwise filtration $(K_i)_{i \in I}$ of a finite simplicial complex $K$.
A persistence pair $(\sigma,\tau)$ of this filtration is an
\emph{emergent facet pair} if $\sigma$ is the youngest facet of $\tau$, 
and an \emph{emergent cofacet pair} if $\tau$ is the oldest cofacet of $\sigma$.
\end{definition}

In other words, $(\sigma,\tau)$ is an emergent facet pair if the column of $\tau$ in the filtration boundary matrix is reduced, and an emergent cofacet pair if the column of $\sigma$ in the coboundary matrix is reduced.
Thus, the columns corresponding to emergent pairs are precisely the ones that are unmodified by the matrix reduction algorithm.

Note that $(\sigma,\tau)$ is an apparent pair if and only if it is both an emergent facet pair and an emergent cofacet pair.
In contrast to the notion of an apparent pair, however, the property of being an emergent pair does depend on the choice of the coefficient field.

\begin{proposition}\label{lem: emergent pairs}
Let $\sigma, \tau$ be simplices in the lexicographically refined Rips filtration.
Then $(\sigma, \tau)$ is a zero persistence emergent cofacet pair if and only if
\begin{itemize}
\item $\tau$ is the lexicographically 
maximal cofacet of $\sigma$ such that $\diam(\tau) = \diam(\sigma)$, and
\item 
$\tau$ does not form a persistence pair $(\rho, \tau)$ with any simplex $\rho$ younger than $\sigma$.
\end{itemize}
A dual statement holds for emergent facet pairs.
\end{proposition}
\begin{proof}
Similar to \cref{prop: zero apparent pairs}, both conditions hold for any zero persistence emergent cofacet pair, and it remains to show that the conditions imply that $(\sigma, \tau)$ is an emergent cofacet pair. 
Again, the first condition implies that $\tau$
is the oldest cofacet of $\sigma$ in the filtration order.
Now if $\tau$ is not paired with some younger simplex $\rho$, we conclude from \cref{prop: persistence pairs} that $(\sigma, \tau)$ forms a persistence pair, which is then an emergent cofacet pair.
\end{proof}

\paragraph{Shortcuts for apparent and emergent pairs}
In practice, zero persistence apparent and emergent pairs provide a way to identify persistence pairs $(\sigma, \tau)$ without even enumerating all cofacets of the simplex~$\sigma$, and thus without fully constructing the coboundary matrix for the reduction algorithm.

We first describe how to determine whether a simplex $\sigma$ appears in an apparent pair $(\sigma, \tau)$. 
Following \cref{prop: zero apparent pairs}, enumerate the cofacets of $\sigma$ in reverse lexicographic order until encountering the first cofacet~$\tau$ with the same diameter as $\sigma$. 
Subsequently, enumerate the facets of $\tau$ in forward lexicographic order, until encountering the first facet of $\tau$ with the same diameter as $\tau$. 
Now $(\sigma,\tau)$ is an apparent pair if and only if that facet equals $\sigma$. 
No further cofacets of $\sigma$ need to be enumerated; this is the mentioned shortcut.
This way, we can determine for a given simplex $\sigma$ whether it appears in an apparent pair $(\sigma,\tau)$, and identify the corresponding cofacet $\tau$.
By an analogous strategy, we can also determine whether a simplex $\tau$ appears in an apparent pair $(\sigma,\tau)$, and identify the corresponding facet $\sigma$.

For the emergent pairs, recall that the columns of the coboundary filtration matrix are processed in reverse filtration order, from youngest to oldest simplex.
Assume that $D$ is the filtration $k$-coboundary matrix, corresponding to the coboundary operator $\delta_k \colon C^k(K) \to C^{k+1}(K)$.
Let $\sigma$ be the current simplex whose column is to be reduced, i.e., the columns of the matrix $R = D \cdot V$ corresponding to younger simplices are already reduced, while the current column is still unmodified, containing the coboundary of $\sigma$.
Following \cref{lem: emergent pairs}, enumerate the cofacets of the simplex $\sigma$ in reverse lexicographic order. 
When we encounter the first cofacet~$\tau$ with the same diameter as $\sigma$, we know that $\tau$ is the youngest cofacet of $\sigma$.
Equivalently, the index of
$\tau$ is the pivot index of the column for $\sigma$ in the coboundary matrix.
Up to this point, all persistence pairs
$(\rho, \phi)$ with simplices $\rho$ younger than $\sigma$ have been identified already, since the algorithm processes simplices from youngest to oldest. 
Thus, if $\tau$ has not previously been paired with any simplex, we conclude from \cref{prop: persistence pairs} that $(\sigma, \tau)$ is an emergent cofacet pair.
Again, no further cofacets of $\sigma$ need to be enumerated, shortcutting the construction of the coboundary column for~$\sigma$.

\section{Implementation}
\label{sec: Implementation}

We now discuss the main data structures and the relevant implementation details of Ripser, the C++ implementation of the algorithms discussed in this paper.
The code is licensed under the MIT license and available at \url{ripser.org}.
The development of Ripser started in October 2015, with support for the emergent pairs shortcut added in March 2016, and support for sparse distance matrices added in September 2016.
The first version of Ripser has been released in July 2016.
The version discussed in the present article (v1.2) has been released in February 2020.

\paragraph{Input}
The input to Ripser is a finite metric space $(X,d)$, encoded in a comma (or whitespace, or other non-numerical character) separated list as either a distance matrix (full, lower, or upper triangular part), or as a list of points in some Euclidean space (\lstinline"euclidean_distance_matrix"), from which a distance matrix is constructed.
The data type for distance values and coordinates is a 32 bit floating point number (\lstinline"value_t").
There are two data structures for storing distance matrices: \lstinline"compressed_distance_matrix" is used for dense distance matrices, storing the entries of the lower (or upper) triangular part of the distance matrix in a \lstinline"std::vector", sorted lexicographically by row index, then column index.
The adjacency list data structure \lstinline"sparse_distance_matrix" is used when the persistence barcode is computed only up to a specified threshold, storing only the distances below that threshold.
If no threshold is specified, the minimum enclosing radius 
\[\min_{x \in X} \max_{y\in X} d(x,y)\]
of the input is used as a threshold, as suggested by \citet{HenselmanPetrusek2017Matroids} and implemented in Eirene.
Above that threshold, for any point $x \in X$ minimizing the above formula for the enclosing radius, the Vietoris--Rips complex is a simplicial cone with apex $x$, and so the homology remains trivial afterwards.

\paragraph{Vertices and simplices}
\label{sec:vertices_simplices}
Vertices are identified with natural numbers $\{0, \dots, n-1\}$, where $n$ is the cardinality of the input space.
Simplices are indexed by natural numbers according to the combinatorial number system.
The data type for both is \lstinline"index_t", which is defined as a 64 bit signed integer (\lstinline"int64_t").
The dimension of a simplex is not encoded explicitly, but passed to methods as an extra parameter.

The enumeration of the vertices of a simplex encoded in the combinatorial number system can be performed efficiently in decreasing order of the vertices.
It is based on the following simple observation.%
\begin{lemma}
Let $N$ be the number of a $d$-simplex $\sigma$ in the combinatorial number system with vertex indices $i_d > \dots > i_0$.
Then the largest vertex index of $\sigma$ is
\(i_d = \max \left\{i\in \mathbb N \mid {i \choose d+1} \leq N\right\}.\)
\end{lemma}
\begin{proof}
First note that ${i_d \choose d+1}$ is a summand of $N=\sum_{l=0}^{d}{i_l \choose {l+1}}$, and so we have\[{i_d \choose d+1} \leq N.\]
For the maximality of $i_d$, note that $\sigma$ is a $d$-simplex on the vertex set $\{0, \dots, i_d\}$. In total, there are ${i_d+1 \choose d+1}$ $d$-simplices on that vertex set, indexed  in the combinatorial number system with the numbers $0,\dots,{i_d+1 \choose d+1}-1$. Thus, we have \[N < {i_d+1 \choose d+1}.
\qedhere\]
\end{proof}
Using this lemma, the largest vertex index of a simplex can be found by performing a binary search, implemented in the method \lstinline"get_max_vertex".
Moreover, the second largest vertex index, $i_{d-1}$, equals the largest vertex index of the simplex with vertex indices $(i_d,\dots,i_0)$, which has number $N - {i_d \choose d+1}$. This way, the numbers of any simplex can be computed by iteratively applying the above lemma. This is implemented in the method \lstinline"get_simplex_vertices".

The binary search for the maximal vertex of a simplex is implemented in \lstinline"get_max_vertex".
The requisite computation of binomial coefficients is done in advance and stored in a lookup table (\lstinline"binomial_coeff_table").

\paragraph{Enumerating cofacets and facets of a simplex}
\label{sec: computing cofacets}
The columns of the coboundary matrix are computed by enumerating the cofacets (implemented in the class \lstinline"simplex_coboundary_enumerator").
There are two implementations, one for sparse and one for dense distance matrices.
For the enumeration of facets there is only one implementation (implemented in the class \lstinline"simplex_coboundary_enumerator"), as every facet of a simplex has to be contained in the filtration by the property that a simplicial complex is closed under the fact relation.

For sparse distance matrices with a distance threshold $t$ (\lstinline"sparse_distance_matrix"), the cofacets of a simplex are obtained by taking the intersection of the neighbor sets for the vertices of the simplex,
\[\bigcap_{x \in \sigma} \{y \in X\mid d(x,y) \leq t\}.\]
More specifically, the distances are stored in an adjacency list data structure, maintaining for every vertex an ordered list of its neighbors together with the corresponding distance, sorted by the indices of the neighbors (\lstinline"std::vector<std::vector<index_diameter_t>> neighbors").
The enumeration of cofacets of a given simplex (\lstinline"ripser<sparse_distance_matrix>::simplex_coboundary_enumerator") searches through the adjacency lists for the vertices of the simplex for common neighbors (in the method (\lstinline"has_next").
Any common neighbor then gives rise to a cofacet of the simplex.

For dense distance matrices, the following straightforward method is used to enumerate the cofacets of a $d$-simplex $\sigma = \{v_{i_{d}},\dots,v_{i_0}\}$ in reverse colexicographic order, represented by their indices in the combinatorial number system.
Enumerating $j = n - 1 , \dots , 0$,
for each $j \notin \{i_{d},\dots, i_0 \}$ there is a unique subindex $k \in \{0,\dots, d\}$ with $i_{k+1} > j > i_{k}$; the corner cases $j > i_{d}$ and $i_0 > j$ are covered by letting $i_{d+1} = n$ and $i_{-1}=-1$.
For the corresponding $k$th cofacet of $\sigma$ we obtain the number
\[(v_{i_{d}}, \dots , v_{i_{k+1}}, v_j, v_{i_k} \dots , v_{i_0}) \mapsto \sum_{l=k+1}^{d}{i_l \choose {l+2}}
+
{j \choose {k+1}}
+ 
\sum_{l=0}^{k}{i_l \choose {l+1}}
.
\]
Thus, the cofacets of $\sigma$ can easily be enumerated in reverse colexicographic vertex order by maintaining and updating the values of the two partial sums appearing in the above equation, starting from the number for the simplex $\sigma$, starting from the value $0$ for the left partial sum (\lstinline"idx_above") and the number for the simplex $\sigma$ for the right partial sum (\lstinline"idx_below").
The coboundary coefficient of the corresponding cofacet is $(-1)^k$.

Returning to the example from \cref{sec: combinatorial numbering system}, the cofacets of the simplex $\{v_5,v_3,v_0\}$ in the full simplex on seven vertices $\{v_0,\dots,v_6\}$ are, in reverse colexicographic vertex order, the simplices with numbers
\begin{align*}
(6,5,3,0) &\textstyle \mapsto 0 + {6 \choose 4} + \left({5 \choose 3} + {3 \choose 2} + {0 \choose 1}\right) = 0 + 15 + (10 + 3 + 0) = 0 + 15 + 13 = 28, \\
(5,4,3,0) &\textstyle \mapsto {5 \choose 4} + {4 \choose 3} + \left({3 \choose 2} + {0 \choose 1}\right) = 5 + 4 + (3 + 0) = 5 + 4 + 3 = 12, \\
(5,3,2,0) &\textstyle \mapsto \left({5 \choose 4} + {3 \choose 3}\right) + {2 \choose 2} + {0 \choose 1} = (5 + 1) + 1 + 0 = 6 + 1 + 0 = 7, \\
(5,3,1,0) &\textstyle \mapsto \left({5 \choose 4} + {3 \choose 3} + {1 \choose 2}\right) + {0 \choose 1} + 0 = (5 + 1 + 0) + 0 + 0 = 6 + 0 + 0 = 6.
\end{align*}

The facets of a $d$-simplex $\sigma = \{v_{i_{d}},\dots,v_{i_0}\}$ can be enumerated in a similar way, this time in forward colexicographic vertex order (implemented in the class \lstinline"simplex_coboundary_enumerator").
Enumerating $k = d + 1 , \dots, 1, 0$ and letting $j=i_k$, for the corresponding $k$th facet of $\sigma$ we obtain the number
\[(v_{i_{d}}, \dots , v_{i_{k+1}}, v_{i_{k-1}}, \dots , v_{i_0}) 
\mapsto 
\sum_{l=k+1}^{d}{i_l \choose {l}}
+ 
\sum_{l=0}^{k-1}{i_l \choose {l+1}}
.
\]
The facets of $\sigma$ can easily be enumerated in (forward) colexicographic vertex order by maintaining and updating the values of the two partial sums appearing in the above equation, starting from the value $0$ for the left partial sum (\lstinline"idx_above") and the number for the simplex $\sigma$ for the right partial sum (\lstinline"idx_below").

\paragraph{Assembling column indices of the filtration coboundary matrix}

The above methods for enumerating cofacets of a simplex are also used to enumerate the $d$-simplices corresponding to the columns in the filtration coboundary matrix, i.e., the essential simplices and the death simplices for relative cohomology (in the method \lstinline"assemble_columns_to_reduce").
A list of the $(d-1)$-simplices is passed as an argument.
Instead of enumerating all cofacets of each $(d-1)$-simplex $\sigma = \{v_{i_{d-1}},\dots,v_{i_0}\}$ in the list of $(d-1)$-simplices, only the cofacets of the form $(v_j, v_{i_{d-1}}, \dots, v_{i_0})$ with $j>i_{d-1}$ are enumerated (using the method \lstinline"has_next" with parameter \lstinline"all_cofacets" set to \lstinline"false"). This way, each $d$-simplex $(v_{i_{d}}, \dots, v_{i_0})$ is enumerated exactly once, as a cofacet of $(v_{i_{d-1}}, \dots, v_{i_0})$.

Since version 1.2, Ripser only assembles column indices that do not correspond to simplices appearing in a zero persistence apparent pair. This strategy, adopted from the parallel GPU implementation of Ripser by \citet{zhang_et_al:LIPIcs:2020:12228}, leads to another significant improvement both in memory usage and in computation time.
The method \lstinline"is_in_zero_apparent_pair") checks whether a simplex has a zero persistence apparent cofacet or facet. Since low-dimensional simplices in Vietoris–Rips filtration more frequently have apparent cofacets than apparent facets, the check for cofacets is carried out first.
To check for a zero persistence apparent cofacet of a simplex (\lstinline"get_zero_apparent_cofacet"), following \cref{prop: zero apparent pairs}, the method \lstinline"get_zero_pivot_cofacet" searches for the lexicographically maximal cofacet with the same diameter as the simplex, and then in turn \lstinline"get_zero_pivot_facet" searches for its lexicographically minimal facet with the same diameter. If that facet is the initial simplex, a zero apparent pair is found.
Checking for an apparent facet (\lstinline"get_apparent_facet") works analogously.

\paragraph{Coefficients}
Ripser supports the computation of persistent homology with coefficients in a prime field~$\mathbb F_p$, for any prime number $p < 2^{16}$. 
The support for coefficients in as prime field can be enabled or disabled by setting a compiler flag (\lstinline"USE_COEFFICIENTS").
The data type for coefficients is \lstinline"coeff_t", which is defined as a 16 bit unsigned integer (\lstinline"uint16_t"), admitting fast multiplication without overflow on 64 bit architectures.
Fast division in modular arithmetic is obtained by precomputing the multiplicative inverses of nonzero elements of the field (in the method \lstinline"multiplicative_inverse_vector").

\paragraph{Column and matrix data sutrctures}
The basic data type for entries in a (\lstinline"diameter_entry_t") boundary or coefficient matrix is a tuple consisting of a simplex index (\lstinline"index_t"), a floating point value (\lstinline"value_t") caching the diameter of the simplex with that index, and a coefficient (\lstinline"coeff_t")
if coefficients are enabled.
The type \lstinline"diameter_entry_t" thus represents a scalar multiple of an oriented $d$-simplex, seen as basis element of the cochain vector space $C^d(K)$.
If support for coefficients is enabled, the index (48 bit) and the coefficient (16 bit) are packed into a single 64 bit word (using \lstinline"__attribute__((packed))").
The actual number of bits used for the coefficients can be adjusted by changing \lstinline"num_coefficient_bits", in order to accommodate a larger number of possible simplex indices.

The reduction matrix $V$ used in the persistence computation is represented as a list of columns in a sparse matrix format (\lstinline"compressed_sparse_matrix"), with each column storing only a collection of nonzero entries, encoding a linear combination of the basis elements for the row space.
The diagonal entries of $V$ are always $1$ and are therefore not stored explicitly in the data structure.
Note that the rows and columns of $V$ are not indexed by the combinatorial number system, but by a prefix of the natural numbers corresponding to a filtration-ordered list of boundary columns (\lstinline"columns_to_reduce").

\paragraph{Pivot extraction}
During the matrix reduction, the current working columns $V_j$ and $R_j$ on which column operations are performed
(\lstinline"working_reduction_column" and \lstinline"working_coboundary")
are maintained as binary heaps (\lstinline"std::priority_queue") with value type \lstinline"diameter_entry_t", using a comparison function object to specify the ordering of the heap elements (\lstinline"greater_diameter_or_smaller_index") in reverse filtration order (of the simplices in the lexicographically refined Rips filtration), thus providing fast access to the pivot entry of a column.

A heap encodes a column vector as a sum of scalar multiples of the row basis elements, the summands being encoded in the data type \lstinline"diameter_entry_t".
The heap may actually contain several entries with the same row index, and should thus be considered as a lazily evaluated representation of a formal linear combination.
In particular, the pivot entry of the column is obtained (in the method \lstinline"pop_pivot") by iteratively extracting the top entries of the heap and summing up their coefficients as long as their row index does not change.
At any point, the coefficient sum might become zero, in which case the procedure continues with the next row index regardless of its row index.
A similar lazy heap data structure has been used already in PHAT \cite{Bauer2017Phat} and DIPHA \cite{Bauer2014Distributed}.

\paragraph{Coboundary matrix}
The method \lstinline"init_coboundary_and_get_pivot" initializes both working columns as $V_j = e_j$ and $R_j = D_j$ and returns the pivot of the column $D_j$.
During the construction of $D_j$, the method checks for a possible emergent pair $(i,j)$ while enumerating the cofacets of the simplex with index $j$.
If the pivot index of $D_j$ is found to form an emergent pair with $j$, the method immediately returns the pivot, without completing the construction of $D_j$.
Since the implicit matrix reduction variant of \cref{Matrix reduction} discards this column afterwards, retaining only the pivot index ($i = \colpivot R j$) and the pivot entry ($\colpivotentry R j$), this shortcut does not affect the correctness of the computation.

\paragraph{Persistence pairs}
The computation of persistence barcodes proceeds by applying \cref{Matrix reduction with clearing} to the filtration coboundary matrix, as described in \cref{Cohomology}.
First, the persistent cohomology in degree $0$ is computed (in \lstinline"compute_dim_0_pairs") using Kruskal's minimum spanning tree algorithm \cite{MR78686} with a 
union-find \cite{MR532171} data structure (\lstinline"union_find").
After that, the remaining barcodes are computed in increasing dimension (\lstinline"compute_dim_0_pairs").

Apparent pairs are not stored in the hash table. Consequently, the check for a column with a given pivot index proceeds in two steps: first querying the hash table, and if no pair is found, checking for an apparent pair (\lstinline"get_apparent_facet").
The method 
\lstinline"add_coboundary"
performs the columns additions
$R_j = R_j - \lambda_j \cdot D \cdot V_k$
and
$V_j = V_j - \lambda_j \cdot V_k$
in \cref{Matrix reduction}.

If the reduction at column index $j$ finishes with a nonzero working coboundary $R_j$ and thus a persistence pair $(i,j)$ is found, the index in the vector \lstinline"columns_to_reduce" corresponding to column~$R_j$ is stored at key~$i$ in the hash table \lstinline"pivot_column_index" (\lstinline"std::unordered_map"), providing fast queries for the column~$j$ with a given pivot index $i = \colpivot R j$, as required in \cref{Matrix reduction}.
The working reduction column $V_j$ is 
written into the compressed reduction matrix, while the working coboundary $R_j$ is discarded.

Since the keys of the hash table are precisely the birth indices of persistence pairs, by the clearing optimization these indices are excluded when assembling the column indices for the coboundary matrix in the next dimension (in the method \lstinline"assemble_columns_to_reduce").
The key type for the hash table is \lstinline"entry_t", and the key in the hash table contains $\colpivot R j$ (in the \lstinline"index_t" member of the type \lstinline"entry_t") as well as $\colpivotentry R j$ (in the \lstinline"coefficient_t" member).
Only the \lstinline"index_t" field is used for hashing and comparing keys.

\section{Experiments}
\label{sec: Experiments}

We compare Ripser (v1.2) to the four most efficient publicly available implementations for the computation of Vietoris--Rips persistence: Dionysus 2 (v2.0.8) \cite{Dionysus}, DIPHA (v2.1.0) \cite{Bauer2014Distributed}, Gudhi (v3.2.0) \cite{GUDHI}, and  Eirene (v1.3.5) \cite{Henselman2016Matroid}.
All results were obtained on a desktop computer with 4 GHz Intel Core i7 processor and 32 GB 1600 MHz DDR3 RAM.
The benchmark is implemented in Docker and can be reproduced using the command \lstinline"docker build github.com/Ripser/ripser-benchmark" on any machine with sufficient memory.
The software DIPHA, written to support parallel and distributed persistence computation, was configured to run on all 4 physical cores.
The data set \emph{sphere3} consists of 192 random points on the unit sphere in $\mathbb R^3$.
It is taken from the benchmark for PHAT~\cite{Bauer2017Phat}.
The data set \emph{o3} consists of 4096 random orthogonal $3 \times 3$ matrices.
For this data set, we computed cohomology up to degree 3 and up to a diameter threshold of $1.4$.
We also used a prefix of the \emph{o3} dataset consisting of 1024 matrices, for which we used the threshold $1.8$.
The data set \emph{torus4} consists of 50000 random points from the Clifford torus $S^1 \times S^1 \subset \mathbb R^4$, for which we used a diameter threshold of $0.15$.
The data sets \emph{dragon}, \emph{fractal-r}, and \emph{random16} are taken from the extensive benchmark \cite{Otter2017Roadmap}.
Our results are shown in \cref{table: benchmark}.

\begin{table}[t]
\centering
\footnotesize
\begin{tabular}{lrrrr@{ }rr@{ }rr@{ }rr@{ }rr@{ }r}
            & $n$ & $p$ & $t$ & \multicolumn{2}{c}{Dionysus}  & \multicolumn{2}{c}{DIPHA}  &  \multicolumn{2}{c}{Gudhi}  & \multicolumn{2}{c}{Eirene}  & \multicolumn{2}{c}{Ripser} \\
\hline
sphere3     &  192& 2 &    & 596\,s,&3.4\,GB        & 26.9\,s,&5.5\,GB     & 46.7\,s,&2.6\,GB     & 11.7\,s,&1.6\,GB     & 0.66\,s,&116\,MB  \\     
o3          & 1024& 3 & 1.8&        &               & 39.1\,s,&1.4\,GB     &  7.6\,s,&610\,MB     &  4.2\,s,&566\,MB     &  2.1\,s,&46.3\,MB \\     
o3          & 4096& 3 & 1.4&        &               &         &            &  344\,s,&17.8\,GB    &  159\,s,&14.9\,GB    & 43.1\,s,&1.1\,GB  \\     
torus4      &50000& 2 &0.15&        &               &         &            &  571\,s,&22.7\,GB    &         &            &  117\,s,&8.0\,GB  \\     
dragon      & 2000& 1 &    &        &               &         &            &         &            & 95.0\,s,&8.0\,GB     &  2.4\,s,&96.7\,MB \\
fractal-r   & 512 & 2 &    &        &               &         &            &         &            & 33.1\,s,&5.6\,GB     &  5.5\,s,&518\,MB  \\
random16    &  50 & 7 &    &        &               &         &            &         &            &  7.6\,s,&1.5 \,GB    &  3.7\,s,&201\,MB
\end{tabular}

\caption{Running times and memory usage of different software packages for various data sets.
The number of points is denoted by $n$, the maximal degree of homology to be computed is denoted by $p$, and the diameter threshold is denoted by $t$.
}
\label{table: benchmark}
\end{table}

\paragraph{Apparent and emergent pairs}

For typical data sets, a large portion of the persistence pairs are apparent or emergent pairs of persistence $0$ and can thus be identified using the shortcuts described in \cref{shortcut}, as shown in \cref{table: pair counts}.
This table shows the counts of various pairs for the data sets in each dimension, starting from $1$.
As predicted by \cref{zero pairs in dim 1 are apparent}, in dimension $1$ every zero pair is an apparent pair and hence also an emergent pair.
However, some non-zero emergent pairs appear as well.
In higher dimensions, there are non-emergent zero persistence pairs.
The speedup obtained by the emergent and apparent pair optimizations is shown in \cref{table: speedup}.

\begin{table}[h]
\centering
\footnotesize
\begin{tabular}{lrrrrrrr}
            & $n$ & $d$ & non-zero pairs & non-emergent & non-shortcut  & non-apparent & total pairs\\
\hline
sphere3     &  192&  1 & 53 & 42 & 53 & 53 & 18\,145   \\
            & & 2 & 1 & 601 & 601 & 3032 & 1\,143\,135   \\
o3          & 1024&  1 & 576 & 451 & 576 & 576  & 32\,167\\     
            &       &  2 & 180 & 2127 & 2149 & 7324 & 325\,428\\     
            &       &  3 & 7 & 8385 & 8385 & 38\,205  & 1\,768\,568\\     
o3          & 4096&  1 &  2466 & 1909& 2466 & 2466 & 230\,051\\     
            &       &  2 &  811 & 15\,734& 15\,829  & 55\,465 & 4\,112\,971 \\     
            &       &  3 &  33 & 123\,031& 123\,035  & 542\,461 & 39\,738\,338\\     
torus4      &50\,000& 1 & 14\,006 & 10\,892& 14\,006 & 14\,006 & 2\,192\,209 \\     
            &       & 2 & 136 & 60\,603& 60\,635 & 263\,403 & 37\,145\,478 \\     
dragon      & 2000& 1 & 576 & 450 & 576 & 576 & 1\,386\,646\\
fractal-r   & 512 & 1 & 438 & 344 & 438 & 438 & 122\,324 \\
            &     & 2 & 659 & 6612 & 6727 & 23\,811 & 18\,979\,158 \\
random16    &  50 & 1 &  39 & 23 & 39 & 39 & 893 \\
            &     & 2 &  16 & 72 & 74 & 227 & 8746 \\
            &     & 3 &  5  & 168 & 169 & 1022 & 54\,026 \\
            &     & 4 &  3  & 297 & 297 & 3455 & 230\,131 \\
            &     & 5 &  0  & 449 & 449 & 8593 & 714\,525 \\
            &     & 6 &  0  & 576 & 576 & 15\,513 & 1\,677\,037 \\
            &     & 7 &  0  & 620 & 620 & 20\,353 & 3\,049\,075
\end{tabular}
\caption{Counts of different types of persistence pairs for the data sets of \cref{table: benchmark}. 
A~\emph{zero} pair has persistence~$0$ in the Vietoris--Rips filtration.
An \emph{emergent} pair corresponds to a column that is already reduced  in the initial boundary matrix.
A \emph{shortcut} pair is both a zero pair and an emergent pair (in particular, also an \emph{apparent} pair).
}
\label{table: pair counts}
\end{table}

\paragraph{Implicit reduction matrix}

Implicit matrix reduction (\cref{sec: Implicit filtration}) is a prerequisite for discarding the columns of the reduced matrix $R = D \cdot V$ instead of storing them in memory.
This, in turn, is a prerequisite for the emergent pairs shortcut, which does not even fully construct all columns of the boundary matrix.
Finally, the apparent pairs shortcut further avoids storing apparent pairs in the pivot hash table and the list of columns to be reduced.
\Cref{table: speedup} exemplifies the speedup obtained by these optimizations.
The running times shown in this table are obtained by making small modification to Ripser to disable the optimizations in the order of their dependencies.

Given the similar timings of implicit reduction (with \emph{storing} the reduced matrix) and of explicit reduction (also \emph{using} the reduced matrix for columns additions), we observe that the extra cost of recreating the reduced columns, as required for implicit reduction, is actually negligible in practice. 

\begin{table}[h]
\centering
\footnotesize
\begin{tabular}{lrrr@{ }rr@{ }rr@{ }rr@{ }rr@{ }r}
            & $n$ & $p$ & \multicolumn{2}{c}{explicit} & \multicolumn{2}{c}{implicit} & \multicolumn{2}{c}{discarded} & \multicolumn{2}{c}{emergent} & \multicolumn{2}{c}{apparent} \\
\hline
sphere3     &    192& 2 & 15.3\,s,&4.1\,GB & 15.6\,s,&4.1\,GB & 5.6\,s,&196\,MB & 1.0\,s,&196\,MB & 0.66\,s,&116\,MB  \\ 
o3          & 1\,024& 3 & 6.9\,s,&1.2\,GB &  7.2\,s,&1.2\,GB & 5.3\,s,&175\,MB & 2.4\,s,&174\,MB &  2.1\,s,&46.3\,MB  \\
o3          & 4\,096& 3 & & & & & 177\,s,&3.7\,GB&  67.1\,s,&3.7\,GB  &  43.9\,s,&1.1\,GB  \\
torus4      &50\,000& 2 & & & & & 232\,s,&8.2\,GB&  140\,s,&8.2\,GB  &  126\,s,&8.0\,GB  \\
dragon      & 2\,000& 1 & & & & & 40.3\,s,&252\,MB & 2.9\,s,&202\,MB  & 2.4\,s,&96.7\,MB  \\
fractal-r   &   512 & 2 & & & & & 191\,s,&1.76\,GB & 13.6\,s,&1.7\,GB  & 5.5\,s,&518\,MB  \\
random16    &    50 & 7 & 17.6\,s,&1.3\,GB& 16.9\,s,&1.3\,GB& 14.0\,s,&331\,MB &  6.2\,s,&331\,MB &  3.8\,s,&201\,MB
\end{tabular}
\caption{Comparison of running times for different optimization enabled in Ripser:
explicit matrix reduction, implicit matrix reduction storing the reduced matrix $R$, implicit matrix reduction discarding the reduced matrix $R$, and implicit matrix reduction discarding the reduced matrix, employing the emergent pairs shortcut during reduction, and skipping all apparent pairs. All variants also employ clearing and compute cohomology.
}
\label{table: speedup}
\end{table}

\section*{Acknowledgements}
The author thanks Gregory Henselman-Petrusek for helpful discussions and for pointing out the use of the smallest enclosing radius as a homology threshold, and Michael Lesnick for invaluable and extensive feedback on an early draft of this article.
The author has no conflicts of interest.
This research has been supported by the DFG Collaborative Research Center SFB/TRR 109 ``Discretization in Geometry and Dynamics''.

\bibliographystyle{abbrvnaturl}
\bibliography{\jobname}

\end{document}